\documentclass[11pt, a4paper]{amsart}

\usepackage{graphicx}      
\usepackage{natbib}        
\usepackage{overpic}
\usepackage{epstopdf}
\usepackage{enumerate}
\usepackage{amssymb}

\newcommand{\inv}{^{-1}}

\newtheorem{theorem}{Theorem}[section]
\newtheorem{lemma}[theorem]{Lemma}
\newtheorem{corollary}[theorem]{Corollary}

\newtheorem{problem}{Problem}

\newtheorem{defin}{Definition}
\newtheorem{exa}{Example}

\newcommand{\C}{\mathbb{C}}						

\newcommand{\matrices}[1]{\mathcal{M}\left( #1 \right)}		

\newcommand{\Hinfty}{H^\infty}								

\newcommand{\error}{e}										
\newcommand{\control}{u}									
\newcommand{\measurement}{y}								
\newcommand{\reference}{\measurement_r}						
\newcommand{\disturbance}{d}								

\newcommand{\Plant}{P}										
\newcommand{\Cont}{C}										
\newcommand{\Gen}{\Theta}									
\newcommand{\gen}{\theta}									
\newcommand{\Closedloop}[1]{H( #1 )}						

\newcommand{\num}{N}										
\newcommand{\den}{D}										
\newcommand{\lnum}{\widetilde{\num}}						
\newcommand{\lden}{\widetilde{\den}}						

\newcommand{\outd}{n}										

\newcommand{\stable}{\mathbf{R}}							
\newcommand{\fractions}[1]{\mathbf{F}} 

\usepackage{mathtools}
\usepackage{mathabx}

\newsavebox{\accentbox}

\begin{document}

\title[A Reformulation of the Internal Model Principle]{Robust Regulation of MIMO systems: A Reformulation of the Internal Model Principle}


\author{Petteri Laakkonen} 

\address{Laboratory of Mathematics, Tampre University of Technology, PO Box 553, FI-33101 Tampere, Finland (e-mail: petteri.laakkonen@tut.fi).}
\maketitle

\begin{abstract}                
The internal model principle is a fundamental result stating a necessary and sufficient condition for a stabilizing controller to be robustly regulating. Its classical formulation is given in terms of coprime factorizations and the largest invariant factor of the signal generator which sets unnecessary restrictions for the theory and its applicability. In this article, the internal model principle is formulated using a general factorization approach and the generators of the fractional ideals generated by the elements of the signal generator. The proposed results are related to the classical ones.
\end{abstract}



\section{Introduction}

The control configuration studied is given in Figure \ref{fig:Closedloop}. The robust regulation problem studied in this article is as follows. Let $\stable$ be an integral domain, and $\fractions{\stable}$ the field of fractions of $\stable$. The plant $\Plant$ and the controller $\Cont$ are matrices over $\fractions{\stable}$ and the reference signal $\reference$ is generated by a signal generator $\Gen_r$ that is a matrix over $\fractions{\stable}$, i.e. $\reference = \Gen_r \measurement_0$ where $\measurement_0$ is a vector over $\stable$. The robust regulation problem aims at finding a controller $\Cont$ such that despite the disturbance signal $d$ and internal perturbations of the plant $\Plant$ the error $\error=\reference+\measurement$ is stable, i.e. a vector over $\stable$. Here the actual reference signal to be tracked is $-\reference$, but the sign convention $+$ is convenient because of the symmetry (see Theorem \ref{thm:Stability}).

\begin{figure}[ht]
\centering
\begin{overpic}[scale=0.75]{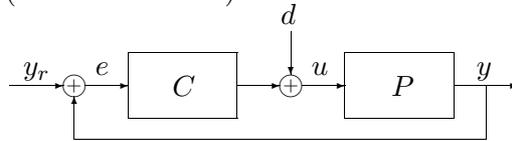}
\put(3,13){$\reference$}
\put(17,13){$\error$}
\put(32,9){$\Cont$}
\put(59,13){$\control$}
\put(53,23){$\disturbance$}
\put(74.5,9){$\Plant$}
\put(91,13){$\measurement$}
\end{overpic}
\caption{The control configuration.}
\label{fig:Closedloop}
\end{figure}

In this paper, a reformulation of the famous internal model principle of robustly regulating controllers by \cite{FrancisWonham1975a} is given. The internal model principle is a necessary and sufficient condition for a stabilizing controller to solve the robust regulation problem, and it states that the instability generated by $\Gen_r$ must be built into every element of $\Cont$. The understanding of this principle leads to internal model based robust controller design techniques studied for example by \cite{HamalainenPohjolainen2000} and \cite{RebarberWeiss2003}.

If $\stable_0$ is the set of all rational functions with complex coefficients that are bounded at infinity and whose poles all have negative real parts, then the controller has a right coprime factorization $\Cont=ND\inv$, i.e. $N$ and $D$ are matrices over $\stable_0$ such that there exist matrices $X$ and $Y$ over $\stable_0$ satisfying the equation $XN+YD=I$. In addition, the signal generator has a left coprime factorization $\Gen_r=D_r\inv N_r$. Let $\alpha\in\stable_0$ be the largest invariant factor of $D_r$. The classical frequency domain formulation of the internal model principle given by \cite{Vidyasagar} states that if $\Cont$ stabilizes $\Plant$, then $\Cont$ solves the robust regulation problem if and only if the elements of $\alpha\inv D$ are in $\stable_0$. The instability of $\Gen_r$ is characterized by the unstable poles, i.e. the poles in the right half plane $\{s\in\C\, |\, \mathrm{Re}(s)\geq 0\}$. The unstable poles are just the zeros of $\alpha$, so the internal model principle forces these poles into every element of the robustly regulating controller.

Frequency domain formulations of the internal model principle for rings that are suitable for infinite dimensional systems are given by \cite{YamamotoHara1988} for pseudo-rational functions, and by \cite{LaakkonenPohjolainen2015} for a stability type that corresponds to polynomial stability in the time domain. A step towards more general robust regulation theory that use the fractional representation approach was taken by \cite{LaakkonenQuadrat2015}. Laakkonen and Quadrat studied the robust regulation of single-input single-output (SISO) systems using fractional ideals, and gave a simple formulation of the internal model principle.

The main result of this paper given by Theorem \ref{thm:IMP} is a reformulation of the internal model principle in terms of the elements of the signal generator and the controller. In its proof, the fractional representations approach presented by \cite{Quadrat2006} is used instead of coprime factorizations. Thus, the internal model principle is extended to integral domains that are not Bezout domains. The main result is the multi-input multi-output (MIMO) extension of the internal model principle of \cite{LaakkonenQuadrat2015}. Corollary \ref{cor:IMP2} shows that the internal model of the signal generator can be understood in terms of the fractional ideal generated by the elements of the signal generator. Finally, Theorem \ref{thm:IMPClassic} shows that in Bezout domains the two formulations are equivalent.


\section{Notations, Preliminary Results and the Problem Formulation}\label{sec:Preliminaries}

A matrix $M$ with elements $\gen_{ij}$ on the $i$th row and $j$th column is denoted by $M=(\gen_{ij})$. We denote the set of all matrices with elements in a set $S$ by $\matrices{S}$ and the set of all $n\times m$ matrices by $S^{n\times m}$. We choose the set of stable elements to be a commutative integral domain $\stable$ that has a unit element. The field of fractions of $\stable$ is denoted by $\fractions{\stable}$. An $\stable$-module $f_1 \stable+\cdots+f_n\stable$, where $f_1,\ldots,f_k\in\fractions{\stable}$, is denoted by $\langle f_1,\ldots,f_k\rangle$ or $\langle f_i\, |\, i=1,\ldots, k\rangle$.

\begin{defin}\begin{enumerate}
\item An $\stable$-submodule $I$ of $\fractions{\stable}$ is called \emph{a fractional ideal} if there exists $0\neq a\in\stable$ such that $a I\subseteq\stable$.
\item A fractional ideal $I$ is \emph{finitely generated} if $I=\langle f_1,\ldots,f_k\rangle$ for some $f_1,\ldots,f_k\in\fractions{\stable}$ and it is
\emph{principal} if it is generated by a single element, i.e. $I=\langle f\rangle$ for some $f\in\fractions{\stable}$.
\end{enumerate}

\end{defin}

The plant and the controller are matrices over $\fractions{\stable}$. It follows that the closed loop of Figure \ref{fig:Closedloop} has a matrix representation as well.
\begin{defin}
\begin{enumerate}
\item A matrix or a vector $H\in\matrices{\fractions{\stable}}$ is \emph{stable} if $H\in\matrices{\stable}$, and otherwise it is unstable.

\item A controller $\Cont\in\fractions{\stable}^{m\times n}$ \emph{stabilizes} $\Plant\in\fractions{\stable}^{n\times m}$ if the closed loop system of Figure \ref{fig:Closedloop} from $(\reference,\disturbance)$ to $(\error,\control)$ given by
\begin{align*}
\Closedloop{\Plant,\Cont} & :=
\begin{bmatrix}
\left(I-\Plant\Cont\right)\inv  & \left(I-\Plant\Cont\right)\inv \Plant\\
\Cont\left(I-\Plant\Cont\right)\inv  &\left(I-\Cont\Plant\right)\inv 
\end{bmatrix}
\end{align*}
is stable.

\end{enumerate}
\end{defin}

\begin{defin}
\begin{enumerate}
\item The representation $\Gen=\num\den\inv $ ($\Gen=\lden\inv \lnum$) is called a right (left) factorization of $\Gen$ if $\num,\den\in\matrices{\stable}$ ($\lnum,\lden \in\matrices{\stable}$) and $\det(\den)\neq 0$ ($\det(\lden)\neq 0$).
\item A factorization $\Gen=\num\den\inv $ ($\Gen=\lden\inv \lnum$) is called \emph{a right (left) coprime factorization} of $\Gen$ if there exist $X,Y\in\matrices{\stable}$ ($\widetilde{X},\widetilde{Y}\in\matrices{\stable}$) such that
$$
X\num+Y\den=I \qquad (\lnum \widetilde{X}+ \lden\widetilde{Y}=I).
$$
\end{enumerate}
\end{defin}

Theory developed in this article is based on the stability results of \cite{Quadrat2006}. The first item of the next theorem is Theorem 3 of \cite{Quadrat2006} and gives a parametrization of all stabilizing controllers. The second item is obtained from the first one by changing the roles of $\Plant$ and $\Cont$ by the symmetry of the closed loop control configuration of Figure \ref{fig:Closedloop}.
\begin{theorem}
\label{thm:Stability}
Let $\Cont$ stabilize $\Plant$.
\begin{enumerate}[1.]
\item Denote
\begin{align*}
\widetilde{L} :=\begin{bmatrix}
-\left(I-\Cont\Plant\right)\inv \Cont\;\; & \left(I-\Cont\Plant\right)\inv
\end{bmatrix},
\end{align*}
and
\begin{align*}
L :=\begin{bmatrix}
\left(I-\Plant\Cont\right)\inv \\
\Cont\left(I-\Plant\Cont\right)\inv 
\end{bmatrix}.
\end{align*}
All stabilizing controllers of $\Plant$ are parametrized by
\begin{subequations}
\begin{align}\label{eqn:AllStabilizingControllers}
\Cont(W) & =\left(\Cont\left(I-\Plant\Cont\right)\inv +\widetilde{L}WL\right)\notag\\
& \qquad \times\left(\left(I-\Plant\Cont\right)\inv +\Plant \widetilde{L}WL\right)\inv \\
& =\left(\left(I-\Cont\Plant\right)\inv +\widetilde{L}WL\Plant\right)\inv \notag\\
&\qquad \times\left(\left(I-\Cont\Plant\right)\inv +\widetilde{L}WL\right)
\end{align}
\end{subequations}
where $W$ is a stable matrix with suitable dimensions such that it satisfies $\det\left(\left(I-\Plant\Cont\right)\inv +\Plant \widetilde{L}WL\right)\neq 0$ and $\det\left(\left(I-\Cont\Plant\right)\inv +\widetilde{L}WL\Plant\right)\neq 0$.

\item Denote
\begin{align}\label{eq:ErrorTM}
\widetilde{M} :=\begin{bmatrix}
-\left(I-\Plant\Cont\right)\inv \Plant\;\; & \left(I-\Plant\Cont\right)\inv
\end{bmatrix},
\end{align}
and
\begin{align}\label{eq:ErrorTM0}
M :=\begin{bmatrix}
\left(I-\Cont\Plant\right)\inv \\
\Plant\left(I-\Cont\Plant\right)\inv 
\end{bmatrix}.
\end{align}
All plants that $\Cont$ stabilizes are parametrized by
\begin{subequations}
\begin{align}\label{eqn:AllStabilizedPlants}
\Plant(X) & =\left(\Plant\left(I-\Cont\Plant\right)\inv +\widetilde{M}XM\right)\notag\\
&\qquad \times\left(\left(I-\Cont\Plant\right)\inv +\Cont \widetilde{M}XM\right)\inv \\
& =\left(\left(I-\Plant\Cont\right)\inv +\widetilde{M}XM\Cont\right)\inv \notag\\
&\qquad \times\left(\left(I-\Plant\Cont\right)\inv +\widetilde{M}XM\right)
\end{align}
\end{subequations}
where $X$ is a stable matrix with suitable dimensions such that it satisfies $\det\left(\left(I-\Cont\Plant\right)\inv +\Cont \widetilde{L}XL\right)\neq 0$ and $\det\left(\left(I-\Plant\Cont\right)\inv +\widetilde{L}XL\Cont\right)\neq 0$.
\end{enumerate}
\end{theorem}

We make the standing assumption that all the reference signals are generated by some fixed signal generators $\Gen_r\in\fractions{\stable}^{n\times q}$, i.e. they are of the form $\reference=\Gen_r\measurement_0$ where the vector $\measurement_0\in\stable^{q\times 1}$. In this article we concentrate on the regulation, so we assume that the disturbance signals contain only unstable dynamics that are already present in the signal generator. In other words, we assume that the disturbance signals are of the form $\disturbance=\Gen_d\disturbance_0$ where the vector $\disturbance_0\in\stable^{q\times 1}$ and $\Gen_d=Q\Gen_r\in\fractions{\stable}^{n\times q}$ for some fixed matrix $Q\in\stable^{m\times n}$.

\begin{defin}
\begin{enumerate}
\item We say that a controller $\Cont\in\fractions{\stable}^{m\times n}$ is \emph{regulating} for $\Plant\in\fractions{\stable}^{n\times m}$ if for all $\measurement_0\in\stable^{q\times 1}$
\begin{align*}
\left(I-\Plant\Cont\right)\inv \Gen_r\measurement_0\in \matrices{\stable}.
\end{align*}

\item We say that a controller $\Cont$ is \emph{disturbance rejecting} for $\Plant$ if for all $\disturbance_0\in\stable^{q\times 1}$
\begin{align*}
\left(I-\Plant\Cont\right)\inv \Plant \Gen_d\disturbance_0\in \matrices{\stable},
\end{align*}
\item A controller $\Cont$ \emph{robustly regulates} $\Plant$ if 
\begin{enumerate}[i)]
\item it stabilizes $\Plant$, and
\item regulates every plant it stabilizes.
\end{enumerate}
\item A controller $\Cont$ is \emph{robustly disturbance rejecting} for $\Plant$ if 
\begin{enumerate}[i)]
\item it stabilizes $\Plant$, and
\item is disturbance rejecting for every plant it stabilizes.
\end{enumerate}
\end{enumerate}
\end{defin}

\begin{problem}
We call the problem of finding a controller $\Cont$ that robustly regulates and is robustly disturbance rejecting for a given nominal plant $\Plant$ \emph{the robust regulation problem.}
\end{problem}

\section{The Internal Model Principle}\label{sec:IMP}

The main result of this paper is the formulation of the internal model principle given by the following theorem. It states a necessary and sufficient condition for a stabilizing controller to be robustly regulating. It generalizes Theorem 3.1 of \cite{LaakkonenQuadrat2015} to multi-input multi-output systems. It is a reformulation of the well-known result that all the unstable dynamics produced by the signal generator must be built into the controller as an internal model in order to make it robustly regulating.

\begin{theorem}\label{thm:IMP}
Denote $\Gen_r=(\gen_{ij})$. Controller $\Cont$ solves the robust regulation problem for $\Plant$ if and only if it stabilizes $\Plant$ and for all $1\leq i\leq \outd$ and $1\leq j\leq q$ there exist $A_{ij},B_{ij}\in\matrices{\stable}$ such that 
\begin{align}\label{eqn:RORPSolvability}
\gen_{ij} I=A_{ij}+B_{ij}\Cont.
\end{align}
\end{theorem}

The proof of the theorem is divided into two lemmas. Lemma \ref{lem:IMP1} shows the sufficiency and Lemma \ref{lem:IMP2} the necessity. The proof of the necessity uses only reference signals, so we see that the internal model is required even if there is no disturbance signals. On the other hand, if there is no reference signals, but the disturbance signals contain unstable dynamics the above condition may be too strong (\cite{LaakkonenPohjolainen2015}).

\begin{exa}
For SISO plants, Theorem \ref{thm:IMP} takes the form $\langle \Gen_r\rangle\subseteq \langle 1,\Cont\rangle$, see \cite{LaakkonenQuadrat2015}. The inclusion indicates that the signals generated by the generator can be divided into a stable part and an unstable part generated by the controller.
\end{exa}

\begin{lemma}\label{lem:IMP1}
Denote $\Gen=(\gen_{ij})$ and let $\Cont$ stabilize $\Plant$. The controller $\Cont$ solves the robust regulation problem if for all $1\leq i\leq \outd$ and $1\leq j\leq q$ there exist $A_{ij},B_{ij}\in\matrices{\stable}$ such that 
$ 
\gen_{ij} I=A_{ij}+B_{ij}\Cont.
$
\end{lemma}
\begin{proof}
Using the notation \eqref{eq:ErrorTM}, then 
\begin{align*}
\widetilde{M} \begin{bmatrix}
\reference \\ \disturbance
\end{bmatrix}&
=\begin{bmatrix}
\left(I-\Plant\Cont\right)\inv  & \left(I-\Plant\Cont\right)\inv \Plant
\end{bmatrix}
\begin{bmatrix}
\reference \\ \disturbance
\end{bmatrix}
\in\matrices{\stable}
\end{align*}
for all the reference and disturbance signals and any plant $\Plant$ that $\Cont$ regulates is equivalent to $\Cont$ solving the robust regulation problem. Since
\begin{align*}
\reference=\sum_{i,j}\gen_{ij}\measurement_{ij} \qquad\text{and}\qquad \disturbance=\sum_{i,j}\gen_{ij}\disturbance_{ij}
\end{align*}
where $\measurement_{ij}$ and $\disturbance_{ij}$ are arbitrary stable vectors, it is sufficient to show that
\begin{align*}
\gen_{ij}\widetilde{M}
\in\matrices{\stable}
\end{align*}
for all $1\leq i\leq \outd$ and $1\leq j\leq q$. Since $\gen_{ij}I=A_{ij}+B_{ij}\Cont$, it follows that
\begin{align*}
\gen_{ij}\widetilde{M}
 =A_{ij}\widetilde{M}+B_{ij}\widetilde{M}\Cont\in\matrices{\stable}.
\end{align*}
On the right hand side of the equation $\widetilde{M}$ and $\widetilde{M}\Cont$ are stable since $\Cont$ stabilizes $\Plant$.\hfill$\square$
\end{proof}

\begin{lemma}\label{lem:IMP2}
Denote $\Gen=(\gen_{ij})$ and let $\Cont$ stabilize $\Plant$. If the controller $\Cont$ is robustly regulating for $\Plant$ then  for all $1\leq i\leq \outd$ and $1\leq j\leq q$ there exist $A_{ij},B_{ij}\in\matrices{\stable}$ such that 
$ 
\gen_{ij} I=A_{ij}+B_{ij}\Cont.
$
\end{lemma}
\begin{proof}

Assume that $\Cont$ robustly regulates $\Plant$. First it is shown that $\theta_{ij}(I-\Plant\Cont)\inv \in\stable$. For the rest of the proof the notation $\widetilde{M}_1=(I-\Plant\Cont)\inv$ is used. The matrix $\widetilde{M}_1$ is stable, since $\Cont$ stabilizes $\Plant$. Since $\Cont$ regulates all the plants it stabilizes, the second item of Theorem \ref{thm:Stability} implies that
\begin{align}
& (I-\Plant(X)\Cont)\inv \Gen_r\notag\\
& \qquad =\left(\widetilde{M}_1 +\widetilde{M}X M\Cont\right.\notag\\
& \qquad\;\;\; \left.-\left(\widetilde{M}_1\inv \Plant+\widetilde{M}X M\right)\Cont\right)\inv\left(\widetilde{M}_1 +\widetilde{M}X M\Cont\right)\Gen_r\notag\\
& \qquad= \left(\widetilde{M}_1 -\widetilde{M}_1 \Plant\Cont\right)\inv \left(\widetilde{M}_1 +\widetilde{M}X M\Cont\right)\Gen_r\notag
\\
& \qquad= \widetilde{M}_1\Gen_r +\widetilde{M}X M\Cont\Gen_r\in\matrices{\stable},\label{eq:RegStabPert}
\end{align}
where $X$ is an arbitrary matrix of suitable dimension and $\widetilde{M}$ and $M$ are given by \eqref{eq:ErrorTM} and \eqref{eq:ErrorTM0}, respectively. Choosing $X=0$ yields $\widetilde{W}_1\Gen_r\in\matrices{\stable}$. This and \eqref{eq:RegStabPert} imply that $\widetilde{M}X M\Cont\Gen_r\in\matrices{\stable}$. In particular,
\begin{align*}
& \widetilde{M}\begin{bmatrix}
0 & 0\\
0 & X_0
\end{bmatrix} M\Cont\Gen_r\\
&\qquad  = 
\begin{bmatrix}
-\widetilde{M}_1\Plant & \widetilde{M}_1
\end{bmatrix}
\begin{bmatrix}
0 & 0\\
0 & X_0
\end{bmatrix}
\begin{bmatrix}
(I-\Cont\Plant)\inv \\ \Plant(I-\Cont\Plant)\inv
\end{bmatrix}
\Cont\Gen_r\\
&\qquad = \widetilde{M}_1 X_0 \Plant(I-\Cont\Plant)\inv\Cont\Gen_r\\
&\qquad = \widetilde{M}_1 X_0 (I-\Plant\Cont)\inv\Plant\Cont\Gen_r\in\matrices{\stable}\\
&\qquad = \widetilde{M}_1 X_0 (\widetilde{M}_1-I)\Gen_r\in\matrices{\stable}.
\end{align*}
Since $\widetilde{M}_1\Gen_r\in\matrices{\stable}$, it follows that
$\widetilde{M}_1X_0\Gen_r\in\matrices{\stable}$ for an arbitrary matrix $X_0$. Letting $X_0$ vary over all matrices of appropriate size shows that
\begin{align*}
\gen_{ij}\widetilde{M}_1=\gen_{ij}(I-\Plant\Cont)\inv\in\matrices{\stable}.
\end{align*}
Similar arguments show that
\begin{align*}
\gen_{ij}(I-\Plant\Cont)\inv\Plant \in\matrices{\stable}.
\end{align*}
The proof is completed by choosing the stable matrices $A_{ij}=\gen_{ij}(I-\Plant\Cont)\inv$ and $B_{ij}=\gen_{ij}(I-\Plant\Cont)\inv\Plant$ and observing that
\begin{align*}
\gen_{ij} I=\gen_{ij} (I-\Plant\Cont)\inv(I-\Plant\Cont)=A_{ij}+B_{ij}\Cont.
\end{align*}\hfill$\square$
\end{proof}

Theorem \ref{thm:IMP} shows that the instability implied by any element $\gen_{ij}$ of the signal generator must be built into every element of a robustly regulating controller. This is the general formulation of the internal model principle. Checking the condition \eqref{eqn:RORPSolvability} for every $\gen_{ij}$ separately is not always needed. The overall instability captured by all the elements of $\Gen_r$ is often characterized by a smaller set of elements. The following corollary makes this statement precise.

\begin{corollary}\label{cor:IMP2}
Denote $\Gen_r=(\gen_{ij})$ and let $\Cont$ stabilize $\Plant$. Consider the fractional ideal $I=\langle \gen_{ij} | 1\leq i\leq n, 1\leq j\leq q\rangle$.
\begin{enumerate}
\item If $I\subseteq \langle f_1,\ldots,f_k\rangle$ and there exist $A_l$ and $B_l$ such that $f_l I=A_l+B_l \Cont$ for all $l=1,\ldots, k$, then $\Cont$ is robustly regulating. 
\item If $\langle f_1,\ldots,f_k\rangle\subseteq I$ and $\Cont$ is robustly regulating, then there exist $A_l$ and $B_l$ such that $f_l I=A_l+B_l \Cont$ for all $l=1,\ldots, k$.
\end{enumerate}
\end{corollary}
\begin{proof}
Only the first item is shown. The second item can be shown similarly. It is assumed that $I\subseteq \langle f_1,\ldots,f_k\rangle$ and that there exist $A_l$ and $B_l$ such that $f_l I=A_l+B_l \Cont$ for all $l=1,\ldots, k$. Now $\gen_{ij}\in \langle f_1,\ldots,f_k\rangle$ or equivalently
\begin{align*}
\gen_{ij}=a_1 f_1+\cdots +a_k f_k
\end{align*}
for some $a_1,\ldots, a_k\in \stable$. Consequently
\begin{align*}
\gen_{ij} I & =\sum_{l=1}^k a_l f_l I\\
& = \sum_{l=1}^k a_l(A_l+B_l \Cont)\\
& =\left(\sum_{l=1}^k a_l A_l\right)+\left(\sum_{l=1}^{k}a_l B_l\right)\Cont.
\end{align*}
Since $\gen_{ij}$ is an arbitrary element of $\Gen_r$, the result follows by Theorem \ref{thm:IMP}.\hfill$\square$
\end{proof}

The above corollary shows that the instability generated by $\Gen_r=(\gen_{ij})$ is captured by the fractional ideal $I$ generated by the elements $\gen_{ij}$. In particular, if $I$ is principal, i.e. there exists an element $\gen\in\fractions{\stable}$ such that $I=\langle\gen\rangle$, then a stabilizing controller is robustly regulating if and only if there exist stable $A$ and $B$ such that
\begin{align*}
\gen I=A+B\Cont.
\end{align*}
Every finitely generated fractional ideal of $\fractions{\stable}$ is principal if and only if $\stable$ is a Bezout domain. Thus, if $\stable$ is a Bezout domain the internal model to be built into a robustly regulating controller is characterized by a single element of $\fractions{\stable}$.

\begin{exa}\label{exa:example1}
The set of all rational functions with complex coefficients that are bounded at infinity and whose poles all have negative real parts is a principal ideal domain, and consequently a Bezout domain. Thus, the internal model is always captured by a single rational function in the field of fractions.

Other common rings in systems theory, e.g. the Hardy space $\Hinfty$ of bounded holomorphic functions in the right half plane $\C_+=\{s\in\C\, |\, \mathrm{Re}(s)>0 \}$ and the convolution algebra $\mathcal{A}(\beta)$ presented by \cite{CallierDesoer1978}, are not typically Bezout. Then there are signal generators for which the instability it generates is not captured by any single fraction over the ring.
\end{exa}

\begin{theorem}\label{thm:IMPClassic}
Let $\Cont$ stabilize $\Plant$ and assume that the fractional ideal $I=\langle \gen_{ij} | 1\leq i\leq n, 1\leq j\leq q\rangle$ generated by the elements of $\Gen_r=(\gen_{ij})$ is principal with the generator $\gen\in\fractions{\stable}$. If $\gen=\frac{n}{d}$ is a coprime factorization, then $\Cont$ is robustly regulating if and only if
there exist stable $A_0$ and $B_0$ such that
\begin{align}\label{eq:IMPClassic1}
d\inv I=A_0+B_0\Cont.
\end{align}
If in addition $\Cont$ has a right coprime factorization $\Cont=N D^{-1}$,
then it is robustly regulating if and only if $D=d D_0$ for some $D_0\in\matrices{\stable}$.
\end{theorem}
\begin{proof}
First it is show that $\Cont$ is robustly regulating if and only if \eqref{eq:IMPClassic1} holds. Corollary \ref{cor:IMP2} implies that $\Cont$ is robustly regulating if and only if for some stable $A$ and $B$
\begin{align}\label{eq:IMPClassic2}
\gen I=A+B\Cont.
\end{align}
Multiplying both sides of \eqref{eq:IMPClassic1} by $n$ shows that \eqref{eq:IMPClassic1} implies \eqref{eq:IMPClassic2}, so it remains to show that \eqref{eq:IMPClassic2} implies \eqref{eq:IMPClassic1}. Since $\gen=\frac{n}{d}$ is a coprime factorization there exist $x,y\in\stable$ such that $nx+dy=1$. By using \eqref{eq:IMPClassic2}, one gets
\begin{align*}
d\inv I & =\frac{nx+dy}{d}I \\
& =x\gen I+ yI\\
& =x(A+B\Cont)+yI\\
&=(xA+yI)+(xB)\Cont.
\end{align*}
The remaining part of the theorem is shown by proving that \eqref{eq:IMPClassic1} is equivalent to that $D=dD_0$ for some $D_0\in\matrices{\stable}$. Since it is now assumed that $\Cont=N D^{-1}$ is a coprime factorization, there exist $X,Y\in\matrices{\stable}$ such that $XN+YD=I$. If \eqref{eq:IMPClassic1} holds, then
\begin{align*}
d\inv D=(A_0+B_0\Cont)D=A_0 D+B_0 N:=D_0\in\matrices{\stable},
\end{align*}
or equivalently $D=dD_0$. On the other hand, if $D=d D_0$, then
\begin{align*}
d\inv I=D_0 D\inv=D_0(XN+YD)D\inv=D_0 Y+D_0 X \Cont,
\end{align*}
which completes the proof. \hfill$\square$
\end{proof}

If $I=\langle \gen_{ij} | 1\leq i\leq n, 1\leq j\leq q\rangle$ is principal and its generator has a coprime factorization $\gen=\frac{n}{d}$, then the internal model to be build into a robustly regulating controller is the stable element $d$ by the above theorem. It can be shown using Corollary \ref{cor:IMP2} that $d$ is unique up to multiplication by a unit. In this sense, one has a minimal internal model. By the first item of Corollary \ref{cor:IMP2}, one may choose $d$ to be the internal model even if $n$ and $d$ are not coprime. However, then $d$ is not minimal, since $d\inv$ produces stronger instability than $\Gen_r$ is able to generate, or in other words $I\subsetneq\langle d\inv\rangle$.

Furthermore, $d$ must divide all elements of the denominator of a coprime factorization of the controller, provided that it exists. By Theorems 7.8 and 7.9 of \cite{Lang2002}, $d$ actually is the largest invariant factor of the denominator $D$ of the coprime factorization of $\Gen_r$. This shows that Theorem \ref{thm:IMPClassic} corresponds to Lemma 7.5.8 of \cite{Vidyasagar}, i.e. Theorem \ref{thm:IMP} is a reformulation of the classical internal model principle.

\begin{exa}
It is now assumed that $\stable$ is the set of all rational functions with complex coefficients that are bounded at infinity and whose poles all have negative real parts. Consider the stable plant
\begin{align*}
\Plant(s)=\begin{bmatrix}
 \frac{2}{s+1} & \frac{1}{(2s+1)(s+1)} \\
\frac{1}{(s+1)^2} & \frac{1}{s+1}
\end{bmatrix}
\end{align*}
which is the transfer function matrix of the linearized plant of a quadruple tank laboratory process presented by \cite{Johansson2000}.

Next it is shown that the controller
\begin{align*}
\Cont(s)=\begin{bmatrix}
-\frac{4s^2 + 2s + 2}{s(s^2+1)} &  -\frac{4s^2 + 3s + 5}{5s(s^2+1)} \\
-\frac{s^2 + s + 1}{s(s^2+1)} & -\frac{2s^2 + s + 1}{s(s^2+1)}
\end{bmatrix}
\end{align*}
solves the robust regulation problem for the plant $\Plant(s)$ and the signal generator
\begin{align*}
\Gen_r(s)=\begin{bmatrix}
\frac{1}{s} & -\frac{1}{s^2-1}\\
\frac{1}{s} & \frac{s+2}{s+1}
\end{bmatrix}.	
\end{align*}
A straightforward calculation shows that $\Cont(s)$ stabilizes $\Plant(s)$. In order to verify that the controller contains an internal model, note that
\begin{align*}
& \frac{1}{s}=\frac{(s+1)^3}{s(s^2-1)}\underbrace{\frac{s^2-1}{(s+1)^3}}_{\in\stable},\\ 
& \frac{1}{s^2-1}=\frac{(s+1)^3}{s(s^2-1)}\underbrace{\frac{s}{(s+1)^3}}_{\in\stable},\\
& \frac{s+2}{s+1}=\frac{(s+1)^3}{s(s^2-1)}\underbrace{\frac{s(s-1)(s+2)}{(s+1)^4}}_{\in\stable},\; \text{and}\\
&
\frac{(s+1)^3}{s(s^2-1)}=\frac{1}{s}\underbrace{\frac{2s+1}{s+1}}_{\in\stable}+\frac{1}{s^2+1}\underbrace{\frac{4s}{s+1}}_{\in\stable}+\frac{s+2}{s+1}.
\end{align*}
This means that the fractional ideal generated by the elements of $\Gen_r$ has generator $\gen(s)=\frac{(s+1)^3}{s(s^2+1)}$. Since
\begin{align*}
\gen I & =\underbrace{\gen(s)(I-\Plant(s)\Cont(s))\inv}_{\in\stable}\\
& \qquad +\underbrace{\gen(s)(I-\Plant(s)\Cont(s))\inv\Plant(s)}_{\in\stable} \Cont(s),
\end{align*}
Corollary \ref{cor:IMP2} shows that the controller is robustly regulating.

Note that $\gen\inv\in\stable$, so $1/\gen\inv$ is a coprime factorization of $\gen$. The controller has the right coprime factorization
\begin{align*}
\Cont(s)& =\begin{bmatrix}
-\frac{4s^2 + 2s + 2}{(s+1)^3} &  -\frac{4s^2 + 3s + 5}{5(s+1)^3} \\
-\frac{s^2 + s + 1}{(s+1)^3} & -\frac{2s^2 + s + 1}{(s+1)^3}
\end{bmatrix}\\
& \qquad
\times\begin{bmatrix}
\frac{s(s^2+1)}{(s+1)^3} &  0 \\
0 & \frac{s(s^2+1)}{(s+1)^3}
\end{bmatrix}\inv,
\end{align*}
and as suggested by Theorem 3.5
\begin{align*}
\gen(s)\begin{bmatrix}
\frac{s(s^2+1)}{(s+1)^3} &  0 \\
0 & \frac{s(s^2+1)}{(s+1)^3}
\end{bmatrix}\in\matrices{\stable}.
\end{align*}
\end{exa}


\section{Concluding Remarks}\label{sec:Conclusions}

A new formulation of the classical internal model principle was given as the main result of this paper. It generalizes the classical formulation to non-Bezout integral domains and the SISO formulation by \cite{LaakkonenQuadrat2015} to MIMO plants. The fractional representation approach was used to formulate the internal model principle. Alternative algebraic approaches have a great potential to provide new insights into the robust regulation, see for example \cite{LaakkonenQuadrat2015}. Prominent frameworks for studying robust regulation are the lattice approach by \cite{Quadrat2006} and the geometric systems theory (\cite{Falb1999}) among others, and future research includes finding new formulations of the internal model principle using these frameworks.


\bibliographystyle{plain} 

\end{document}